\documentclass[12pt]{amsart}
\usepackage{amssymb}
\usepackage[all]{xy}

\theoremstyle{remark}
\newtheorem{example}{Example}[section]
\newtheorem{rem}[example]{Remark}
\theoremstyle{definition}
\newtheorem{defn}[example]{Definition}
\theoremstyle{plain}
\newtheorem{prop}[example]{Proposition}
\newtheorem{lemma}[example]{Lemma}

\newtheorem{thm}[example]{Theorem}

\DeclareMathOperator{\image}{image}
\DeclareMathOperator{\Hom}{Hom}
\DeclareMathOperator{\End}{End}

\DeclareMathOperator{\Ext}{Ext}
\DeclareMathOperator{\SU}{SU}
\DeclareMathOperator{\GL}{GL}
\DeclareMathOperator{\PGL}{PGL}
\DeclareMathOperator{\Stab}{Stab}
\DeclareMathOperator{\Grass}{Gr}
\DeclareMathOperator{\rank}{rank}

\begin{document}
\bibliographystyle{plain}

\newcommand{\Cpx}{\mathsf{Cpx}}
\newcommand{\bbZ}{\mathbb{Z}}
\newcommand{\bbC}{\mathbb{C}}
\newcommand{\bbH}{\mathbb{H}}
\newcommand{\bbA}{\mathbb{A}}
\newcommand{\bbP}{\mathbb{P}}
\newcommand{\longto}[1][]{\stackrel{#1}{\longrightarrow}}
\newcommand{\dual}{\mathrm{dual}}
\newcommand{\free}{\mathrm{free}}
\newcommand{\univ}{\mathrm{univ}}
\newcommand{\reg}{\mathrm{reg}}
\newcommand{\pr}{\mathrm{pr}}
\newcommand{\id}{\mathrm{id}}
\newcommand{\calM}{\mathcal{M}}
\newcommand{\calX}{\mathcal{X}}
\newcommand{\calE}{\mathcal{E}}
\newcommand{\EPoinc}{\calE^{\text{Poincar\'{e}}}}
\newcommand{\Euniv}{\calE^{\univ}}
\newcommand{\stackV}{\mathcal{V}}
\newcommand{\stackU}{\mathcal{U}}
\newcommand{\calO}{\mathcal{O}}
\newcommand{\Gm}{\mathbb{G}_m}

\title[Special instanton bundles]{Independent parameters for special instanton bundles on $\bbP^{2n+1}$}

\author{Norbert Hoffmann}
\address{Freie Universit\"at Berlin, Institut f\"ur Mathematik, Arnimallee 3, 14195 Berlin, Germany}
\email{norbert.hoffmann@fu-berlin.de}



\date{}

\begin{abstract}
Motivated by Yang-Mills theory in $4n$ dimensions, and generalizing the notion due to Atiyah, Drinfeld, Hitchin and Manin for $n=1$, 
Okonek, Spindler and Trautmann introduced instanton bundles and special instanton bundles as certain algebraic vector bundles of rank $2n$
on the complex projective space $\bbP^{2n+1}$. The moduli space of special instanton bundles is shown to be rational.
\end{abstract}

\maketitle

\section{Introduction}

An instanton is a self-dual solution of the $\SU( 2)$ Yang-Mills equations on the Euclidean sphere $S^4$ \cite[Exp. 1]{asterisque}.
Via the Penrose transformation, these correspond to certain algebraic vector bundles of rank $2$ on the complex projective space $\bbP^3$,
which are consequently called instanton bundles \cite{ADHM, asterisque}. 

Salamon \cite{salamon} and Corrigan-Goddard-Kent \cite{corrigan-goddard-kent} have independently generalized this picture to Yang-Mills theory in dimension $4n$,
replacing $S^4$ and $\bbP^3$ by the quaternionic projective space $\bbH \bbP^n$ and the twistor space over it, which is the complex projective space $\bbP^{2n+1}$.
Now certain Yang-Mills connections on $\bbH \bbP^n$ correspond to certain algebraic vector bundles on $\bbP^{2n+1}$; cf. \cite[\S 3]{capria-salamon} and
\cite[Corollary to Main Theorem 2]{kametani-nagatomo}. Motivated by this generalization of the Penrose transform,
Okonek and Spindler \cite{okonek-spindler} extended the notion of instanton bundles to algebraic vector bundles of rank $2n$ on $\bbP^{2n+1}$.

A fundamental tool in the study of algebraic vector bundles on projective spaces is their description in terms of monads
\cite{horrocks,barth-hulek,okonek-schneider-spindler}. It often allows to describe the vector bundles at hand in terms of some linear algebra data.
In particular, the Beilinson spectral sequence \cite{beilinson} yields a correspondence between instanton bundles and instanton monads;
we recall this in Section \ref{sec:instantons} below.

Hirschowitz and Narasimhan \cite{hirschowitz-narasimhan} have introduced certain special instanton bundles on $\bbP^3$, calling them special `t Hooft bundles.
They have studied their moduli space, and proved in particular that it is rational. Recently Tikhomirov \cite{tikhomirov} has shown that the moduli space of
all instanton bundles on $\bbP^3$ with fixed odd $c_2$ is irreducible, and together with Markushevich also that it is rational \cite{markushevich-tikhomirov}.

Generalizing from $\bbP^3$ to $\bbP^{2n+1}$, Spindler and Trautmann \cite{spindler-trautmann} have introduced the notion of special instanton bundles
on $\bbP^{2n+1}$ and constructed their moduli space. They also prove that this moduli space is non-empty, irreducible, and smooth of dimension
$2n(k+1) + (2n+2)^2 - 7$ if one fixes the quantum number $k+1$ of these instantons. The main result of the present paper, Theorem \ref{mainthm},
states that the Spindler-Trautmann moduli space of special instanton bundles on $\bbP^{2n+1}$ is rational.
So these instanton bundles depend on $2n(k+1) + (2n+2)^2 - 7$ independent complex parameters.

Like in the case of vector bundles on a curve \cite{king-schofield, rationality, par}, the proof involves the rationality of some Severi-Brauer varieties,
whose Brauer classes are related to the existence of Poincar\'e families, or universal families, of vector bundles parameterized by the moduli space.
Spindler and Trautmann determined in \cite{spindler-trautmann} when such Poincar\'e families of special instanton bundles exist. Somewhat surprisingly,
Theorem \ref{mainthm} proves rationality of the moduli space even in the cases where there is no Poincar\'e family. A similar phenomenon has been observed
for moduli spaces of vector bundles on rational surfaces \cite{schofield, costa-miroroig}.

Another main ingredient in the proof is the no-name lemma about rationality of quotients of vector spaces modulo linear groups;
cf. for example \cite{colliotthelene-sansuc}. It allows us to finally reduce to a rationality problem for quotients modulo $\PGL_2$, where the invariant ring
is known explicitly; that's how we prove rationality of a moduli space without at the same time constructing a Poincar\'e family on some open part of it.
Compared to the the special case $n = 1$ of bundles on $\bbP^3$, where the rationality is proved in \cite{hirschowitz-narasimhan}, this ingredient is new.
Then some results about Severi-Brauer varieties from \cite{hirschowitz-narasimhan} complete the proof.

The structure of the present text is as follows. In Section \ref{sec:instantons}, we recall the definition of mathematical instanton bundles
that we will work with, and we review the correspondence to instanton monads. In Section \ref{sec:moduli}, we recall the notion of special instanton bundles
due to Spindler and Trautmann, and we also review their construction of moduli spaces for these; both aspects are based on the correspondence to instanton monads.
Section \ref{sec:rationality} is devoted to the proof of the main result, Theorem \ref{mainthm}.

\subsubsection*{Acknowledgements}
I learned about this problem from Laura Costa. I thank her, Rosa Maria Mir\'{o}-Roig, and Alexander Schmitt for introducing me into the subject
and for some helpful suggestions. The work was supported by the SFB 647: Raum - Zeit - Materie.

\raggedbottom

\section{Mathematical instanton bundles} \label{sec:instantons}
In this section, we recall the notion of (mathematical) instanton bundles, and their description in terms of monads.

We work over the odd-dimensional complex projective space
\begin{equation*}
  \bbP := \bbP^{2n+1} \big/ \bbC, \quad n \geq 1.
\end{equation*}
An algebraic vector bundle $E$ over $\bbP$ is said to have \emph{natural cohomology} if there is at most one $q$ with $H^q( \bbP, E) \neq 0$.
\begin{defn}
  Let $k \geq 1$ be an integer. A \emph{$k$-instanton bundle} is an algebraic vector bundle $E$ over $\bbP = \bbP^{2n+1}$ with the following properties:
  \begin{itemize}
   \item[i)] The rank of $E$ is $2n$.
   \item[ii)] The Chern polynomial of $E$ is $c_t( E) = (1 - t^2)^{-k}$.
   \item[iii)] For each $l \in \bbZ$ with $-2n-1 \leq l \leq 0$, the twisted vector bundle
    \begin{equation*}
      E( l) := E \otimes \calO_{\bbP}( 1)^{\otimes l}
    \end{equation*}
    has natural cohomology.
  \end{itemize} 
\end{defn}
\begin{rem}
  i) Every $k$-instanton bundle is simple by \cite[Theorem 2.8]{ancona-ottaviani}.

  ii) Some authors moreover require that $E$ has trivial splitting type. Note that this is an open condition.
  It is not used here, so we don't include it in our definition.
\end{rem}
Let $\Cpx( \bbP)$ denote the category of complexes of coherent $\calO_{\bbP}$-modules. By a \emph{monad}, we mean a complex of vector bundles
\begin{equation*}
  E^{\bullet} = [ 0 \longto E^{-1} \longto[ i] E^0 \longto[ p] E^1 \longto 0 ] \in \Cpx( \bbP)
\end{equation*}
such that $p$ is surjective, and $i$ is an isomorphism onto a subbundle. As \emph{morphisms} of monads, we take the morphisms in $\Cpx( \bbP)$.
The \emph{cohomology} of the monad $E^{\bullet}$ is the vector bundle
\begin{equation*}
  E := \ker( p)/\image( i).
\end{equation*}
\begin{defn}
  A monad $E^{\bullet}$ is a \emph{$k$-instanton monad} for $k \geq 1$, if
  \begin{equation*}
    E^{-1} \cong \calO_{\bbP}( -1)^k, \qquad E^0 \cong \calO_{\bbP}^{2n+2k}, \qquad \text{and}\qquad E^1 \cong \calO_{\bbP}( 1)^k.
  \end{equation*}
\end{defn}
The following standard facts show that the categories of $k$-instanton bundles and of $k$-instanton monads are equivalent.
\begin{prop} \label{prop:equivalence} \begin{itemize}
 \item[i)] If $E^{\bullet}$ is a $k$-instanton monad, then its cohomology $E = \ker( p)/\image( i)$ is a $k$-instanton bundle.
 \item[ii)] If $E^{\bullet}$ and $F^{\bullet}$ are two $k$-instanton monads with cohomologies $E$ and $F$,
  then $\Hom_{\bbP}( E, F) = \Hom_{\Cpx( \bbP)}( E^{\bullet}, F^{\bullet})$.
 \item[iii)] Every $k$-instanton bundle $E$ is isomorphic to the cohomology of some $k$-instanton monad $E^{\bullet}$.
\end{itemize} \end{prop}
\begin{proof}
  i) Let $E$ be the cohomology of a $k$-instanton monad
  \begin{equation*}
    0 \longto \calO_{\bbP}( -1)^k \longto[ i] \calO_{\bbP}^{2n+2k} \longto[ p] \calO_{\bbP}( 1)^k \longto 0.
  \end{equation*}
  Then the short exact sequences of vector bundles
  \begin{gather*}
    0 \longto \ker( p) \longto \calO_{\bbP}^{2n+2k} \longto[ p] \calO_{\bbP}( 1)^k \longto 0,\\
    0 \longto \calO_{\bbP}( -1)^k \longto[ i] \ker( p) \longto E \longto 0
  \end{gather*}
  show that $E$ has rank $2n$ and Chern polynomial $c_t( E) = (1 - t^2)^{-k}$.
  Now tensor these sequences with $\calO_{\bbP}( l)$, and consider the resulting long exact cohomology sequences.
  For $l = 0$, we get an exact sequence
  \begin{equation*}
    H^0( \ker(p) ) \longto H^0( E) \longto H^1( \calO_{\bbP}( -1)^k) = 0.
  \end{equation*}
  Here $H^0( \ker(p) ) = 0$, since $\ker( p)$ is stable by \cite[Theorem 2.8]{ancona-ottaviani}; hence $H^0( E) = 0$.
  For general $l$, we see that $H^q( E( l))$ vanishes whenever
  \begin{equation*}
    H^{q-1}( \calO_{\bbP}( l + 1)) = H^q( \calO_{\bbP}( l)) = H^{q+1}( \calO_{\bbP}( l - 1)) = 0.
  \end{equation*}
  The latter holds in each of the following four cases:
  \begin{itemize}
   \item $l = -2n-1$, and $q \neq 2n$
   \item $-2n \leq l \leq -2$, and $q$ arbitrary
   \item $l = -1$, and $q \neq 1$
   \item $l = 0$, and $q \geq 2$
  \end{itemize}
  It follows that $E( l)$ has natural cohomology for $-2n-1 \leq l \leq 0$, the only possibly nonzero cohomology groups in this range being
  \begin{equation*}
    H^{2n}( E( -2n-1)), \quad H^1( E( -1)), \quad\text{and}\quad H^1( E).
  \end{equation*}

  ii) Let $E^{\bullet}$ and $F^{\bullet}$ more generally be two monads, with cohomologies $E$ and $F$. Given a morphism $E \to F$, we try to lift it
  to a morphism of monads $E^{\bullet} \to F^{\bullet}$. As explained in \cite[Lemma II.4.1.3]{okonek-schneider-spindler}, the obstructions
  against the existence and uniqueness of such a lift are some classes in
  \begin{equation*}
    \Ext^q( E^i, F^j) = H^q( \bbP, (E^i)^{\dual} \otimes F^j)
  \end{equation*}
  with $i > j$ and $q \leq 2$. If $E^{\bullet}$ and $F^{\bullet}$ are $k$-instanton monads, then
  \begin{equation*}
    \Ext^q( E^i, F^j) = H^q( \bbP, \calO_{\bbP}( j - i)^{\rank( E_i) \cdot \rank( E_j)})
  \end{equation*}
  vanishes for all $i > j$ and all $q \leq 2$.

  iii) For any vector bundle $E$ over $\bbP$, Beilinson \cite{beilinson} has constructed a spectral sequence with $E_1$-term
  \begin{equation*}
    E_1^{pq} = H^q( \bbP, E( p)) \otimes \Omega_{\bbP}^{-p}( -p)
  \end{equation*}
  which converges to $E$. If $E$ is a $k$-instanton bundle, then most of these terms vanish, and the claim follows;
  for some more details, see for example \cite[Lemma 1.3 and Corollary 1.4]{okonek-spindler}\footnote{In \cite{okonek-spindler},
  all instanton bundles are assumed to be symplectic. However, this assumption is not used in the quoted Lemma 1.3 and Corollary 1.4.}.
\end{proof}
\begin{rem}
  Let $S$ be a scheme of finite type over $\bbC$.

  A vector bundle $\calE$ over $\bbP \times S$ is a \emph{family of $k$-instanton bundles}
  if its restriction to every closed point $s \in S( \bbC)$ is a $k$-instanton bundle.

  A monad $\calE^{\bullet}$ of vector bundles over $\bbP \times S$ is a \emph{family of $k$-instanton monads}
  if its restriction to every closed point $s \in S( \bbC)$ is a $k$-instanton monad.
  (This implies $\calE^d \cong \calO_{\bbP}( d) \boxtimes F^d$ for vector bundles $F^d$ over $S$.)

  With these definitions, the above equivalence between instanton bundles and instanton monads extends to families; cf. \cite[p. 585f.]{spindler-trautmann}.
\end{rem}

\section{Moduli of special instantons} \label{sec:moduli}

This section recalls the notion of \emph{special} instantons due to Spindler-Trautmann \cite{spindler-trautmann},
and their construction of moduli spaces for these.

Still assuming $k \geq 1$, we fix two complex vector spaces $V, W \cong \bbC^2$.
Let $V^{\otimes r} \twoheadrightarrow S^r V$ denote the $r$th symmetric power of $V$. Then
\begin{align*}
  E^0_{n, k} & := (S^{n+k} V \otimes W)^{\dual} \otimes \calO_{\bbP} \qquad\text{and}\\
  E^1_{n, k} & := (S^k V)^{\dual} \otimes \calO_{\bbP}( 1)
\end{align*}
are vector bundles over $\bbP = \bbP^{2n+1}$, with ranks $2n+2(k+1)$ and $k+1$.

The multiplication $\mu_r: S^r V \otimes S^n V \to S^{n+r} V$ induces a linear map
\begin{equation*}
  \mu_r^*: (S^{n+r} V)^{\dual} \longto (S^r V)^{\dual} \otimes (S^n V)^{\dual}.
\end{equation*}
The choice of a linear isomorphism
\begin{equation*}
  b: (S^n V \otimes W)^{\dual} \longto[ \sim] H^0( \bbP, \calO_{\bbP}( 1))
\end{equation*}
thus determines a morphism of vector bundles
\begin{equation*}
  p_b: E^0_{n, k} \longto E^1_{n, k}
\end{equation*}
such that $H^0( p_b)$ is the composition
\begin{equation*}
  H^0( E^0_{n, k}) \longto[ \mu_k^*] (S^k V)^{\dual} \otimes (S^n V \otimes W)^{\dual} \longto[ b] H^0( E^1_{n, k}).
\end{equation*}
Note that this composition is injective, since $\mu_k$ is surjective.
\begin{defn}
  \begin{itemize}
   \item[i)] A $(k+1)$-instanton monad $E^{\bullet} \in \Cpx( \bbP)$ is called \emph{special} if its stupid truncation
    \begin{equation*}
      \tau^{\geq 0}( E^{\bullet}) := [ E^0 \longto[ p] E^1 ]
    \end{equation*}
    is in $\Cpx( \bbP)$ isomorphic to a complex of the form
    \begin{equation*}
      [ E^0_{n, k} \longto[ p_b] E^1_{n, k}]
    \end{equation*}
    for some isomorphism $b: (S^n V \otimes W)^{\dual} \to H^0( \bbP, \calO_{\bbP}( 1))$.
   \item[ii)] A $(k+1)$-instanton bundle $E$ over $\bbP$ is \emph{special} if $E$ is isomorphic to the cohomology of some special $(k+1)$-instanton monad.
  \end{itemize}
\end{defn}
\begin{rem}
  This definition is equivalent to the original definition given by Spindler and Trautmann, according to \cite[Proposition 4.2]{spindler-trautmann}.

  In particular, all these special instanton bundles are simple, and have trivial splitting type; cf. also \cite[Proposition 4.5]{spindler-trautmann}.
\end{rem}
With $k \geq 1$ still fixed, we consider all complexes of the form
\begin{equation*}
  [ E^0_{n, k} \longto[ p_b] E^1_{n, k}] \in \Cpx( \bbP)
\end{equation*}
as above. The next step is to classify them up to isomorphy in $\Cpx( \bbP)$.

We have an exact sequence of algebraic groups
\begin{equation*}
  1 \longto \Gm \longto[ \iota_n] \GL( V) \times \GL( W) \longto[ \pi_n] \GL( S^n V \otimes W),
\end{equation*}
defined by $\iota_n( \lambda) := ( \lambda \id_V, \lambda^{-n} \id_W)$ and $\pi_n( \alpha, \beta) := S^n \alpha \otimes \beta$.
In particular, $\pi_n$ allows us to identify the $7$-dimensional group
\begin{equation*}
  G_n := \GL( V) \times \GL( W)/\iota_n( \Gm)
\end{equation*}
with a closed subgroup of $\GL( S^n V \otimes W)$.
\begin{prop} \label{prop:X_n}
  Let $k \geq 1$ be given. The homogeneous variety
  \begin{equation*}
    X_n := G_n \backslash \GL( S^n V \otimes W)
  \end{equation*}
  is a coarse moduli space for complexes of the form $[ E^0_{n, k} \longto[ p_b] E^1_{n, k}]$.
\end{prop}
\begin{proof}
  Given an isomorphism $b: (S^n V \otimes W)^{\dual} \to H^0( \bbP, \calO_{\bbP}( 1))$ and an element $g \in \GL( S^n V \otimes W)$, we get another isomorphism
  \begin{equation*}
    g \cdot b := b \circ g^{\dual}: (S^n V \otimes W)^{\dual} \longto[ \sim] H^0( \bbP, \calO_{\bbP}( 1)).
  \end{equation*}
  This defines a simply transitive action of $\GL( S^n V \otimes W)$ on the set of all such isomorphisms $b$.

  Every pair $(\alpha, \beta) \in \GL( V) \times \GL( W)$ yields a commutative diagram
  \begin{equation*} \xymatrix{
    (S^{n+k} V \otimes W)^{\dual} \ar[rr]^-{b \circ \mu_k^*} \ar[d]_{(S^{n+k} \alpha^{-1} \otimes \beta^{-1})^{\dual}}
      && (S^k V)^{\dual} \otimes H^0( \bbP, \calO_{\bbP}( 1)) \ar[d]^{(S^k \alpha^{-1})^{\dual}}\\
    (S^{n+k} V \otimes W)^{\dual} \ar[rr]^-{b' \circ \mu_k^*}
      && (S^k V)^{\dual} \otimes H^0( \bbP, \calO_{\bbP}( 1))
  } \end{equation*}
  with $b' := \pi_n( \alpha, \beta) \cdot b$, and hence an isomorphism of complexes
  \begin{equation*} \xymatrix{
    E^0_{n, k} \ar[rr]^{p_b} \ar[d] && E^1_{n, k} \ar[d]\\
    E^0_{n, k} \ar[rr]^{p_{b'}} && E^1_{n, k}.
  } \end{equation*}

  Conversely, suppose that $[ E^0_{n, k} \longto[ p_b] E^1_{n, k}]$ and $[ E^0_{n, k} \longto[ p_{b'}] E^1_{n, k}]$ are isomorphic in $\Cpx( \bbP)$
  for some isomorphisms $b$ and $b'$. Then there is a pair $(\alpha, \beta) \in \GL( V) \times \GL( W)$
  with $b' = \pi_n( \alpha, \beta) \cdot b$, according to step 2) in the proof of \cite[Proposition 6.1]{spindler-trautmann}.

  Pick one isomorphism $b_0: (S^n V \otimes W)^{\dual} \to H^0( \bbP, \calO_{\bbP}( 1))$.
  Assign to each complex $p_{g \cdot b_0}: [ E^0_{n, k} \to E^1_{n, k}]$ with $g \in \GL( S^n V \otimes W)$
  the moduli point $G_n g$ in the coset space $X_n = G_n \backslash \GL( S^n V \otimes W)$.
  The coset $G_n g$ depends only on the isomorphism class of the complex. This turns $X_n$ into a coarse moduli space for the complexes in question.
\end{proof}
The group $G_n$ acts via its quotient $\PGL( V)$ on $\bbP V := (V \setminus \{0\})/\Gm$. Thus the homogeneous variety $X_n$ carries a natural conic bundle
\begin{equation*}
  C := \bbP V \times^{G_n} \GL( S^n V \times W) \longto X_n
\end{equation*}
with fibers $\bbP V$. The fiberwise symmetric power
\begin{equation*}
  C^{(2n+k)} \longto X_n
\end{equation*}
is a projective bundle with fibers
\begin{equation*}
  (\bbP V)^{(2n+k)} = \bbP H^0( \bbP V, \calO( 2n+k)) = \bbP( S^{2n+k} V)^{\dual}.
\end{equation*}
Note that every element $f \in (S^{2n+k} V)^{\dual}$ induces a linear map 
\begin{equation*}
  \mu_{n+k}^*( f): S^n V \longto (S^{n+k} V)^{\dual}.
\end{equation*}
We form the associated Grassmannian bundle
\begin{equation*}
  \Grass_k( C^{(2n+k)}) \longto X_n,
\end{equation*}
which parameterizes linear subspaces $\bbP U \subseteq (\bbP V)^{(2n+k)}$ of dimension $k$,
or equivalently linear subspaces $U \subseteq ( S^{2n+k} V)^{\dual}$ of dimension $k+1$.
\begin{thm}[Spindler-Trautmann] \label{thm:mudulispace}
  Given $k \geq 1$, let
  \begin{equation*}
    M_{\bbP}( k + 1) \subseteq \Grass_k( C^{(2n+k)}) \longto X_n
  \end{equation*}
  denote the open locus of all linear subspaces $U \subseteq ( S^{2n+k} V)^{\dual}$ such that $\mu_{n+k}^*( f)$ is injective for all $0 \neq f \in U$.
  Then $M_{\bbP}( k + 1)$ is a coarse moduli space for special $(k+1)$-instanton bundles over $\bbP = \bbP^{2n+1}$.
\end{thm}
\begin{proof}
  This statement is Theorem 6.3 in \cite{spindler-trautmann}. For the convenience of the reader, we give an outline of the proof.

  The starting point is that every special $(k+1)$-instanton bundle defines a special $(k+1)$-instanton monad, and hence by truncation a point in $X_n$.
  The remaining part of the monad will then be parameterized by the fiber of $M_{\bbP}( k + 1)$ over this point in $X_n$. 

  To be more specific, note that the choice of an isomorphism
  \begin{equation*}
    b: (S^n V \otimes W)^{\dual} \longto[ \sim] H^0( \bbP, \calO_{\bbP}( 1))
  \end{equation*}
  yields a commutative diagram
  \begin{equation*} \xymatrix{
    (S^{n+k} V \otimes W \otimes S^n V \otimes W)^{\dual} \ar[r]^-{\mu_k^*} \ar[d]_b & (S^k V)^{\dual} \otimes S^2 \big((S^n V \otimes W)^{\dual}\big) \ar[d]^{S^2 b}\\
    \Hom_{\bbP} \big( \calO_{\bbP}( -1), E_{n, k}^0 \big) \ar[r]^{(p_b)_*} & \Hom_{\bbP} \big( \calO_{\bbP}( -1), E_{n, k}^1 \big).
  } \end{equation*}
  The kernel of the horizontal map $\mu_k^*$ consists of all multilinear forms
  \begin{equation*}
    \varphi: S^{n+k} V \otimes W \otimes S^n V \otimes W \longto \bbC
  \end{equation*}
  which satisfy the condition
  \begin{align*}
      - & \varphi( v_1,  \ldots, v_n,  v_{n+1}, \ldots, v_{n+k}, w,  v_1', \ldots, v_n', w')\\
      = & \varphi( v_1', \ldots, v_n', v_{n+1}, \ldots, v_{n+k}, w', v_1,  \ldots, v_n,  w)
  \end{align*}
  for all $v_i, v_j' \in V$ and $w, w' \in W$. Using that $\varphi$ is also symmetric in $v_1, \ldots, v_{n+k}$ and in $v_1',  \ldots, v_n'$ separately,
  it is easy to deduce that $\varphi$ is symmetric in all the $v$'s and alternating in the $w$'s. This proves
  \begin{equation*}
    \ker( \mu_k^*) = ( S^{2n+k} V \otimes \Lambda^2 W)^{\dual}.
  \end{equation*}
  Hence $b$ induces, via the above diagram, an isomorphism
  \begin{equation*}
    ( S^{2n+k} V \otimes \Lambda^2 W)^{\dual} \longto[ \sim] \Hom_{\bbP} \big( \calO_{\bbP}( -1), \ker( p_b) \big);
  \end{equation*}
  cf. also \cite[Proposition 3.2]{spindler-trautmann}. Given a special $(k+1)$-instanton monad
  \begin{equation*}
    E^{\bullet} = [0 \longto \calO_{\bbP}( -1)^{(k+1)} \longto[ i] E_{n, k}^0 \longto[ p_b] E_{n, k}^1 \longto 0] \in \Cpx( \bbP),
  \end{equation*}
  the components $\calO_{\bbP}( -1) \to \ker( p_b)$ of $i$ thus span a linear subspace of $( S^{2n+k} V \oplus \Lambda^2 W)^{\dual}$,
  which corresponds to a linear subspace
  \begin{equation*}
    U \subseteq ( S^{2n+k} V)^{\dual}.
  \end{equation*}
  Since $i$ is an isomorphism onto a subbundle, $U$ has dimension $k+1$, and $\mu_{n+k}^*( f)$ is injective for all $0 \neq f \in U$;
  cf. \cite[(3.7)]{spindler-trautmann}. In this way, the special $(k+1)$-instanton monad $E^{\bullet}$ defines a point in $M_{\bbP}( k + 1)$.

  Conversely, every linear subspace $U \subseteq ( S^{2n+k} V)^{\dual}$ of dimension $k+1$ yields, by means of the above isomorphism, a linear subspace
  \begin{equation*}
    (\Lambda^2 W)^{\dual} \otimes U \subseteq \Hom_{\bbP} \big( \calO( -1), \ker( p_b) \big)
  \end{equation*}
  of dimension $k+1$, and hence a morphism of vector bundles
  \begin{equation*}
    i_U: \calO_{\bbP}( -1)^{(k+1)} \cong (\Lambda^2 W)^{\dual} \otimes U \otimes \calO_{\bbP}( -1) \longto \ker( p_b).
  \end{equation*}
  If $\mu_{n+k}^*( f)$ is injective for all $0 \neq f \in U$, then $i_U$ is an isomorphism onto a subbundle of $\ker( p_b)$, so
  \begin{equation*}
    0 \longto \calO_{\bbP}( -1)^{(k+1)} \longto[ i_U] E_{n, k}^0 \longto[ p_b] E_{n, k}^1 \longto 0
  \end{equation*}
  is a special $(k+1)$-instanton monad. This shows that $M_{\bbP}( k + 1)$ is a coarse moduli scheme for special $(k+1)$-instanton monads,
  and hence also for special $(k+1)$-instanton bundles due to Proposition \ref{prop:equivalence}.
\end{proof}
\begin{rem} \label{rem:families}
  To make this precise using the formalism of moduli functors, one would have to define what a \emph{family} of special $(k+1)$-instanton bundles is,
  say parameterized by a scheme $S$ of finite type over $\bbC$.

  A family of isomorphisms $b: (S^n V \otimes W)^{\dual} \to H^0( \bbP, \calO_{\bbP}( 1))$ is an isomorphism
  of the corresponding trivial vector bundles over $S$; it induces as before a complex of vector bundles
  \begin{equation*}
    [E^0_{n, k} \boxtimes \calO_S \longto E^1_{n, k} \boxtimes \calO_S]
  \end{equation*}
  over $\bbP \times S$. One could define that a family $\calE^{\bullet}$ of $(k+1)$-instanton monads
  is special if its truncation $\tau^{\geq 0}( \calE^{\bullet})$ is \'{e}tale-locally in $S$ isomorphic to a complex of this form,
  and that a family of $(k+1)$-instanton bundles $\calE$ is special if the corresponding family of $(k+1)$-instanton monads is.
  Then the arguments in the above proof of Theorem \ref{thm:mudulispace} extend routinely to families of special instantons.

  Note however that the moduli functor in \cite[p. 585]{spindler-trautmann} is slightly different from this one for non-reduced $S$,
  since they only require that the restriction of $\calE$ to all closed points $s \in S( \bbC)$ is special.
\end{rem}
\begin{rem}
  The moduli space $M_{\bbP}( k+1)$ is non-empty by \cite[3.7]{spindler-trautmann}. It is by construction an irreducible smooth variety of dimension
  \begin{equation*}
    \dim M_{\bbP}( k+1) = 2n(k+1) + \dim X_n = 2n(k+1) + (2n+2)^2 - 7.
  \end{equation*}
\end{rem}

\section{Rationality} \label{sec:rationality}
Let $G$ be a linear algebraic group over $\bbC$. Suppose that $G$ acts on an integral algebraic variety $X$ of finite type over $\bbC$.
We denote by $\bbC( X)^G$ the field of $G$-invariant rational functions on $X$.
\begin{lemma} \label{lemma:slice}
  Suppose that $G$ acts on the integral variety $X'$ over $\bbC$ as well, and that there is an open orbit $Gx' \subseteq X'$ with $x' \in X'( \bbC)$. Then
  \begin{equation*}
    \bbC( X \times X')^G \cong \bbC( X)^{\Stab_G( x')}.
  \end{equation*}
\end{lemma}
\begin{proof}
  This is a special case of the standard `lemma of Seshadri', or `slice lemma'; cf. for example \cite[Theorem 3.1]{colliotthelene-sansuc}.
\end{proof}

The action of $G$ on $X$ is called \emph{almost free} if there is a dense open subvariety $X^0 \subseteq X$ such that
the stabiliser subgroup $\Stab_G( x) \subseteq G$ is trivial for each closed point $x \in X^0( \bbC)$.
\begin{lemma} \label{lemma:free}
  Suppose $V \cong \bbC^2$, and $n \geq 2$. The natural action of $\PGL( V)$ on the Grassmannian $\Grass_2( S^n V \oplus S^n V)$ is almost free.
\end{lemma}
\begin{proof}
  Suppose that $\alpha \in \GL( V)$ represents a nontrivial element in $\PGL( V)$.
  Up to multiplication by $\bbC^*$, and the choice of an appropriate basis for $V$, there are three cases:
  \begin{align*}
    1.) \quad \alpha & = \begin{pmatrix} 1 & 0\\0 & \lambda \end{pmatrix}, \quad \lambda \in \bbC^*, \quad \lambda^r \neq 1 \text{ for all } r\in \{1,\ldots,n\},\\
    2.) \quad \alpha & = \begin{pmatrix} 1 & 0\\0 & \zeta \end{pmatrix}, \quad \zeta \text{ a primitive $r$th root of unity}, \quad 2 \leq r \leq n,\\
    3.) \quad \alpha & = \begin{pmatrix} 1 & 1\\0 & 1 \end{pmatrix}.
  \end{align*}
  In each case, let us estimate the dimension of the fixed point set
  \begin{equation*}
    \Grass_2( S^n V \oplus S^n V)^{\alpha}.
  \end{equation*}

  In case 1.), $S^n V$ decomposes into $1$-dimensional eigenspaces under the semisimple endomorphism $\alpha$, so $S^n V \oplus S^n V$ decomposes into
  $2$-dimensional eigenspaces. Every $2$-dimensional $\alpha$-invariant subspace
  \begin{equation*}
    U \subseteq S^n V \oplus S^n V
  \end{equation*}
  is either one of these eigenspaces, or a direct sum of lines in two of them. The former are parameterized by a finite set, the latter by a finite union
  of products $\bbP^1 \times \bbP^1$. Hence we conclude in this case
  \begin{equation*}
    \dim \Grass_2( S^n V \oplus S^n V)^{\alpha} = 2.
  \end{equation*}
  The union of these fixed loci has dimension $\leq 2 + \dim \PGL( V) = 5$.

  In case 2.), $S^n V$ decomposes into $r$ eigenspaces $(S^n V)_{\chi}$ under the semisimple endomorphism $\alpha$. Their dimension is
  \begin{equation*}
    \dim (S^n V)_{\chi} \leq \lceil (n+1)/r \rceil \leq \lceil (n+1)/2 \rceil \leq (n+2)/2.
  \end{equation*}
  So the eigenspace $(S^n V)_{\chi}^2$ of $S^n V \oplus S^n V$ has dimension $\leq n+2$, and the space of $2$-dimensional subspaces $U \subseteq (S^n V)_{\chi}^2$
  has dimension
  \begin{equation*}
    \dim \Grass_2 \big( (S^n V)_{\chi}^2 \big) \leq 2 n.
  \end{equation*}
  If $\chi_1 \neq \chi_2$ are two eigenvalues of $\alpha$ on $S^n V$, then
  \begin{equation*}
    \dim (S^n V)_{\chi_1} + \dim (S^n V)_{\chi_2} \leq \dim S^n V = n+1.
  \end{equation*}
  So the space of direct sums $U=U_1 \oplus U_2$, where $U_i$ is a line in the eigenspace $(S^n V)_{\chi_i}^2$ of $\alpha$ on $S^n V \oplus S^n V$, has dimension
  \begin{equation*}
    \dim \bbP \big( (S^n V)_{\chi_1}^2 \big) \times \bbP \big( (S^n V)_{\chi_2}^2 \big) \leq 2 \dim( S^n V) - 2 = 2n.
  \end{equation*}
  These two arguments cover all $2$-dimensional $\alpha$-invariant subspaces $U \subseteq S^n V \oplus S^n V$. Hence we can conclude in this case
  \begin{equation*}
    \dim \Grass_2( S^n V \oplus S^n V)^{\alpha} \leq 2n.
  \end{equation*}
  Such $\alpha \in \GL( V)$ yield finitely many $2$-dimensional conjugacy classes in $\PGL( V)$.
  So the union of these fixed loci has dimension $\leq 2n + 2$.

  In case 3.), we can choose a basis $x, y \in V$ with
  \begin{equation*}
    \alpha( x) = x \qquad\text{and}\qquad \alpha( y) = x + y.
  \end{equation*}
  The endomorphism $(\alpha - \id)^r$ of $S^n V$ has kernel
  \begin{equation*}
    \ker (\alpha - \id)^r = x^{n-r+1} \cdot S^{r-1} V
  \end{equation*}
  for $1 \leq r \leq n+1$. In particular, the kernel of $(\alpha - \id)^r$ on $S^n V \oplus S^n V$ has dimension $2r$.
  Suppose that $U \subseteq S^n V \oplus S^n V$ is $2$-dimensional and $\alpha$-invariant. It follows that $\alpha - \id$ is a nilpotent endomorphism of $U$.

  If $\alpha - \id = 0$ on $U$, then $U$ coincides with the kernel $\bbC \cdot x^n \oplus \bbC \cdot x^n$ of $\alpha - \id$ on $S^n V \oplus S^n V$;
  hence $U$ is unique in this situation.

  Otherwise, we have $(\alpha - \id)^2 = 0$ on $U$, and $U = \bbC \cdot u \oplus \bbC \cdot \alpha( u)$ for any element $u \in U$ with $\alpha( u) \neq u$.
  Here $u$ can be any element in the $4$-dimensional kernel of $(\alpha - \id)^2$ on $S^n V \oplus S^n V$ with $\alpha( u) \neq u$.

  Hence we conclude in this case
  \begin{equation*}
    \dim \Grass_2( S^n V \oplus S^n V)^{\alpha} = 4 - 2 = 2.
  \end{equation*}
  This $\alpha \in \GL( V)$ yields one $2$-dimensional conjugacy class in $\PGL( V)$. So the union of these fixed loci has dimension $\leq 2 + 2 = 4$.

  The Grassmannian $\Grass_2( S^n V \oplus S^n V)$ has dimension $4n$. Using the assumption $n \geq 2$, we see that the closure of all these fixed points
  has smaller dimension. On the open complement, $\PGL( V)$ acts freely.
\end{proof}
We say that the field $\bbC( X)^G$ of $G$-invariant rational functions on $X$ is \emph{rational} if it is purely transcendental over the base field $\bbC$.
\begin{example}
  The group $\PGL( V)$ with $V \cong \bbC^2$ acts on the vector space $\End( V)^2$ over $\bbC$ by simultaneous conjugation.
  The action is known to be almost free, and the field of invariants
  \begin{equation} \label{eq:PGL_2}
    \bbC \big( \End( V)^2 \big)^{\PGL( V)}
  \end{equation}
  is rational. In fact, sending $(\alpha_1, \alpha_2) \in \End( V)^2$ to the traces of the five maps
  $\alpha_1, \alpha_2, \alpha_1^2, \alpha_1 \alpha_2, \alpha_2^2 \in \End( V)$ defines an isomorphism
  \begin{equation*}
    \End( V)^2 /\!/ \PGL( V) \longto[ \sim] \bbA^5.
  \end{equation*}
\end{example}
\begin{lemma}[No-name lemma] \label{lemma:no-name}
  Let $G$ act linearly on vector spaces $M$ and $M'$ of finite dimension over $\bbC$.
  If $G$ is reductive, and acts almost freely on $M$, then $\bbC( M \oplus M')^G$ is purely transcendental over $\bbC( M)^G$.
\end{lemma} 
\begin{proof}
  This statement is contained for example in \cite[Corollary 3.8]{colliotthelene-sansuc}.
\end{proof}
\begin{example} \label{example:no-name}
  Let the group $\PGL( V)$ with $V \cong \bbC^2$ act linearly on a finite-dimensional vector space $M'$ over $\bbC$. Then the field of invariants
  \begin{equation*}
    \bbC \big( \End( V)^2 \oplus M' \big)^{\PGL( V)}
  \end{equation*}
  is purely transcendental over the field \eqref{eq:PGL_2}, and hence rational.
\end{example}
\noindent We keep the notation of the previous section, so $V, W \cong \bbC^2$, and
\begin{equation*}
  G_n = \GL( V) \times \GL( W)/\iota_n( \Gm) \quad\text{with}\quad \iota_n( \lambda) := ( \lambda \id_V, \lambda^{-n} \id_W).
\end{equation*}
\begin{prop} \label{prop:X_rational}
  The coarse moduli space $X_n = G_n \backslash \GL( S^n V \otimes W)$ constructed in Proposition \ref{prop:X_n} is rational.
\end{prop}
\begin{proof}
  The function field of $X_n$ is by construction
  \begin{equation*}
    \bbC( X_n) \cong \bbC \big( (S^n V \otimes W)^{2n+2} \big)^{G_n}.
  \end{equation*}
  We start with the special case $n = 1$. The action of $G_1$ on $V \otimes W$ has an open orbit, whose points correspond to linear maps $\psi: W^{\dual} \to V$
  that are bijective. Thus Lemma \ref{lemma:slice} yields
  \begin{equation*}
    \bbC( X_1) \cong \bbC \big( (V \otimes W)^4 \big)^{G_1} \cong \bbC \big( (V \otimes W)^3 \big)^{\Stab_{G_1}( \psi)}.
  \end{equation*}
  The canonical projection $G_1 \twoheadrightarrow \PGL( V)$ restricts to an isomorphism
  \begin{equation*}
    \Stab_{G_1}( \psi) \longto[ \sim] \PGL( V).
  \end{equation*}
  Since the linear isomorphism
  \begin{equation*}
    V \otimes W \xrightarrow{(\psi^{-1})^{\dual}} V \otimes V^{\dual} = \End( V)
  \end{equation*}
  intertwines the action of $\Stab_{G_1}( \psi)$ with that of $\PGL( V)$, we conclude
  \begin{equation*}
    \bbC( X_1) \cong \bbC \big( \End( V)^3 \big)^{\PGL( V)}.
  \end{equation*}
  Using Example \ref{example:no-name}, it follows that $X_1$ is rational.

  For the rest of the proof, we assume $n \geq 2$. The group $\GL( V)$ acts on the $2$-dimensional vector space
  \begin{equation*}
    V( n) := \begin{cases} \bbC^2 \otimes \det^{\frac{n}{2}} V & \text{for $n$ even,}\\ V \otimes \det^{\frac{n-1}{2}} V & \text{for $n$ odd,} \end{cases}
  \end{equation*}
  in such a way that the center $\Gm \subseteq \GL( V)$ acts with weight $n$. Thus the action of $\GL( V) \times \GL( W)$ on the $4$-dimensional vector space 
  \begin{equation*}
    V( n) \otimes W
  \end{equation*}
  descends to an action of $G_n$. This action again has an open orbit, whose points correspond to bijective linear maps $\psi: W^{\dual} \to V( n)$.
  
  Viewing elements of $(S^n V \otimes W)^2$ as linear maps from $W^{\dual}$ to the direct sum $S^n V \oplus S^n V$,
  the open locus of injective linear maps is a $\GL( W)$-torsor over the Grassmannian $\Grass_2( S^n V \oplus S^n V)$.
  Thus $G_n$ acts almost freely on $(S^n V \otimes W)^2$, due to Lemma \ref{lemma:free}.

  It follows that the field $\bbC( X_n)$ is purely transcendental over the field
  \begin{equation*}
    \bbC \big( (S^n V \otimes W)^2 \oplus \End( V)^2 \oplus (V( n) \otimes W) \big)^{G_n},
  \end{equation*}
  since both are purely transcendental over $\bbC( (S^n V \otimes W)^2)^{G_n}$ according to Lemma \ref{lemma:no-name},
  the former of transcendence degree $2n(2n+2) \geq 24$, and the latter of transcendence degree $12$.

  Thus Lemma \ref{lemma:slice} yields that $\bbC( X_n)$ is purely transcendental over
  \begin{equation*}
    \bbC \big( (S^n V \otimes W)^2 \oplus \End( V)^2 \big)^{\Stab_{G_n}( \psi)}.
  \end{equation*}
  But this field is rational according to Example \ref{example:no-name}, since the stabiliser of $\psi$ in $G_n$ projects again isomorphically onto $\PGL( V)$.
\end{proof}
\begin{rem} \label{rem:obstruction}
  The conic bundle $C \to X_n$ with fibers $\bbP V$ is not Zariski-locally trivial; cf. \cite[Proposition 8.5]{spindler-trautmann}.
  It follows that the $G_n$-torsor $\GL( S^n V \otimes W) \to X_n$ is not Zariski-locally trivial either. 
  The obstruction against both is a Brauer class, which can be described as follows:

  The proof of Proposition \ref{prop:X_rational} shows that $\bbC( X_n)$ is actually purely transcendental over $\bbC( \End( V)^2)^{\PGL( V)}$.
  Over the latter, one has the so-called generic quaternion algebra; cf. \cite[Chapter 14]{saltman}. Its image in the Brauer group of $\bbC( X_n)$
  is the obstruction class in question.

  An equivalent way to state this is to say that the stack quotient of $\GL( S^n V \otimes W)$ modulo $\GL( V) \times \GL( W)$
  is birational to an affine space times the stack quotient of $\End( V)^2$ modulo $\GL( V)$.
  This can be proved along the same lines as Proposition \ref{prop:X_rational} above.
\end{rem}
\begin{rem} \label{rem:Poincare}
  The variety $X_n$, its rationality, and the local nontriviality of bundles over it in Remark \ref{rem:obstruction}, did not depend on $k$.
  But the existence of Poincar\'{e} families over $\bbP \times X_n$ does depend on $k$, as follows.

  Fix $k \geq 1$. The closed points $x \in X_n( \bbC)$ correspond to certain isomorphism classes of complexes
  \begin{equation*}
    E^{\bullet} = [E^0 \longto E^1] \in \Cpx( \bbP)
  \end{equation*}
  with $E^0 \cong \calO_{\bbP}^{2n+2(k+1)}$ and $E^1 \cong \calO_{\bbP}( 1)^{k+1}$. A complex of vector bundles
  \begin{equation*}
    \calE^{\bullet} = [\calE^0 \longto \calE^1] \in \Cpx( \bbP \times X_n) 
  \end{equation*}
  is a \emph{Poincar\'{e} family} if for every closed point $x \in X_n( \bbC)$, the corresponding isomorphism class contains the restriction of $\calE^{\bullet}$
  to $\bbP \times \{x\}$.

  There is a universal family $b^{\univ}$ of isomorphisms $b$, parameterized by $\GL( S^n V \otimes W)$; cf. Remark \ref{rem:families}.
  It induces a complex of vector bundles
  \begin{equation*}
    [E^0_{n, k} \boxtimes \calO_{\GL( S^n V \otimes W)} \xrightarrow{ p_{b^{\univ}}} E^1_{n, k} \boxtimes \calO_{\GL( S^n V \otimes W)}]
  \end{equation*}
  over $\bbP \times \GL( S^n V \otimes W)$. On this complex, $\GL( V) \times \GL( W)$ acts; the image of $\Gm$ under $\iota_n$ acts with weight $-k$.

  If $k$ is even, we can tensor the complex with the $1$-dimensional representation $\det^{k/2}( V)$ of $\GL( V) \times \GL( W)$.
  After that, the image of $\Gm$ under $\iota_n$ acts trivially, so the action of $\GL( V) \times \GL( W)$ descends to an action of $G_n$.
  Then the complex descends to a complex $\calE^{\bullet}$ over $\bbP \times X_n$, which is indeed a Poincar\'{e} family in the above sense.

  If $k$ is odd, then one can show that no Poincar\'{e} family $\calE^{\bullet}$ over $\bbP \times X_n$ exists; cf. \cite[Proposition 8.5]{spindler-trautmann}.

  One way to view this is to note that the moduli stack parameterizing the complexes in question does depend on $k$.
  It is in fact the stack quotient $\calX_{n, k}$ of $\GL(S^n V \otimes W)$ modulo the group $\GL( V) \times \GL( W)/\iota_n( \mu_k)$,
  which can be proved along the same lines as Proposition \ref{prop:X_n}.

  The stack $\calX_{n, k}$ is birational to an affine space times the stack quotient of $\End( V)^2$ modulo $\GL( V)/\mu_k$; cf. the previous remark.
  Thus the obstruction against Poincar\'{e} families is $k$ times the Brauer class over $\bbC( X_n)$ coming from the generic quaternion algebra.
  So we see again that the obstruction vanishes if and only if $k$ is even.
\end{rem}
\begin{thm} \label{mainthm}
  For $n \geq 1$ and $k \geq 1$, the coarse moduli space $M_{\bbP}( k+1)$ of special $(k+1)$-instanton bundles over $\bbP = \bbP^{2n+1}$ is rational.
\end{thm}
\begin{proof}
  Recall that $M_{\bbP}( k+1)$ is constructed as an open subvariety in the Grassmannian bundle $\Grass_k( C^{(2n+k)})$ over $X_n$.
  We have just seen that $X_n$ is rational. Using \cite[Proposition 2.2]{hirschowitz-narasimhan},
  it follows that $\Grass_k( C^{(2n+k)})$ is also rational, since $k$ and $2n+k$ have the same parity. 
\end{proof}
\begin{rem}
  In the special case $n = 1$, which means $\bbP = \bbP^3$, the rationality of $M_{\bbP}( k+1)$
  has already been proved by Hirschowitz and Narasimhan \cite[Th\'{e}or\`{e}me 4.10]{hirschowitz-narasimhan}.
  They also prove the rationality of $X_1$ \cite[Th\'{e}or\`{e}me 3.4.II]{hirschowitz-narasimhan}; their proof is different from the one given here.
\end{rem}
\begin{rem}
  What \cite[Proposition 2.2]{hirschowitz-narasimhan} proves is that the Grassmannian bundle $\Grass_k( C^{(2n+k)}) \to X_n$ has rational generic fiber.
  This does not necessarily mean that it is Zariski-locally trivial.
  One can show that it is Zariski-locally trivial if and only if $k$ is even; cf. Remark \ref{rem:Poincare}.
\end{rem}
\begin{rem}
  \cite[Theorem 8.2]{spindler-trautmann} states that there is a Poincar\'{e} family of special instanton bundles $\calE$ over $\bbP \times M_{\bbP}( k+1)$
  if and only if $k$ is even. In fact the obstruction class is the pullback of the obstruction class on $X_n$ explained in Remark \ref{rem:Poincare}.

  An equivalent way to state this is to say that the moduli stack $\calM_{\bbP}( k+1)$ of special instanton bundles is birational
  to an affine space times the stack $\calX_{n, k}$ in Remark \ref{rem:Poincare}. Observing that $\calM_{\bbP}( k+1)$ is open
  in a Grassmannian bundle over $\calX_{n, k}$, this can also be proved using Lemma 5.5 and Lemma 4.10 in \cite{par}.
\end{rem}

\bibliography{inst}
\end{document}